\documentclass[11pt]{amsart}

\usepackage{amsmath}
\usepackage{amstext}
\usepackage{amssymb}
\usepackage{amsthm}
\usepackage{enumerate}
\usepackage{bbm}
\usepackage{comment}
\usepackage[USenglish]{babel}
\usepackage[utf8]{inputenc}
\usepackage[T1]{fontenc}
\usepackage{accents}
\usepackage[shortlabels]{enumitem}
\usepackage{todonotes}
\usepackage[all]{xy}

\usepackage[pdfauthor={Vera Fischer, David Schrittesser},
            pdftitle={Sacks indestructible co-analytic eventually different families},
            pdfproducer={Latex with hyperref}]{hyperref}

\makeatletter
\def\imod#1{\allowbreak\mkern10mu({\operator@font mod}\,\,#1)}
\makeatother

\theoremstyle{plain}
\newtheorem{thm}{Theorem}
\theoremstyle{definition}
\newtheorem{defi}[thm]{Definition}

\newtheorem*{cor*}{Corollary}
\newtheorem*{lem*}{Lemma}
\newtheorem{lem}[thm]{Lemma}

\newtheorem{prop}[thm]{Proposition}

\newtheorem{clm}[thm]{Claim}
\newtheorem{fct}[thm]{Fact}
\newtheorem{rem}[thm]{Remark}

\DeclareMathOperator{\dom}{dom}
\DeclareMathOperator{\ran}{ran}

\DeclareMathOperator{\supp}{supp}

\DeclareMathOperator{\ZF}{ZF}

\newcommand{\bbP}{\mathbb{P}}
\newcommand{\bbQ}{\mathbb{Q}}
\newcommand{\bbR}{\mathbb{R}}
\newcommand{\bbN}{\mathbb{N}}

\newcommand{\bbS}{\mathbb{S}}

\newcommand{\ned}[1][{}]{\mathrm{ned}_{#1}}
\newcommand{\codeforinvofe}{d}
\newcommand{\diagbij}{\delta}

\renewcommand{\phi}{\varphi}

\newcommand{\Power}{\mathcal{P}}

\newcommand{\forces}{\Vdash}

\renewcommand{\iff}{\leftrightarrow}

\newcommand{\concatB}{\mathbin{\rotatebox[origin=c]{90}{\scalebox{.7}{(\kern1ex)}}}}

\providecommand{\conc}{ \mathbin{{}^\frown}}
\providecommand{\res}{\mathbin{\upharpoonright} }
\DeclareMathOperator{\lh}{lh}
\DeclareMathOperator{\Vee}{\mathbf{V}}
\DeclareMathOperator{\eL}{\mathbf{L}}
\DeclareMathOperator{\parfin}{parfin}

\begin{document}

\title[A Sacks indestructible co-analytic \emph{med} family]{A Sacks indestructible co-analytic maximal eventually different family}
%on maximal eventually different families
%small co-analytic med families
%co-analytic indestructible med families

% Information for first author
\author{Vera Fischer}
\address{Kurt G\"odel Research Center, University of Vienna, W\"ahringer Stra\ss e 25, 1090 Vienna, Austria}
\email{vera.fischer@univie.ac.at}

% Information second author
\author{David Schrittesser}
\address{Kurt G\"odel Research Center, University of Vienna, W\"ahringer Stra\ss e 25, 1090 Vienna, Austria}
\email{david.schrittesser@univie.ac.at}

%Thanks
%\thanks{The authors would like to thank the Austrian Science Fund (FWF) for the generous support through Grants Numbers Y1012-N35 (Vera Fischer) and F .... (David Schrittesser) }

% General info
\date{\today}
%----------------------------------------------------------------------------------------------------------

\begin{abstract}
In the constructible universe, we construct a co-analytic maximal family of pairwise eventually different functions from $\bbN$ to $\bbN$ which remains maximal after adding arbitrarily many Sacks reals (by a countably supported iteration or product).
\end{abstract}

\maketitle

\section{Introduction}\label{s.intro}

Suppose $\mathfrak X$ is a set and $\mathcal R \subseteq {}^{<\omega}\mathfrak X$ is a relation;
then $\mathcal A \subseteq \mathfrak X$ is called \emph{discrete} (with respect to $\mathcal R$) if and only
if ${}^{<\omega}\mathcal A \cap \mathcal R = \emptyset$, and \emph{maximal discrete} if and only if
in addition, $\mathcal A$ is not a proper subset of an discrete family.

Maximal discrete families always exist by the Axiom of Choice; a \emph{definable} such family may or may not exist. Firstly this depends on the $\mathfrak X$ and $\mathcal R$ in question: 
For example, letting $\mathfrak X = \bbR$, a maximal independent set with respect to linear dependence over $\bbQ$ is usually called a \emph{Hamel basis}; such a set cannot be analytic. 
The same is true for other examples such as \emph{maximal almost disjoint} (short: \emph{mad}) families or \emph{maximal independent families}. 
See \cite{mathias.1967} and \cite[10.28]{miller-infinite}. 

In stark contrast, consider the following binary relation on Baire space ${}^\omega\omega$ (the space of functions from $\bbN$ to $\bbN$ with the product topology): $f$ and $g$ from ${}^\omega\omega$ are \emph{eventually different} if and only if 
\[
(\exists n\in\bbN) ( \forall m \in \bbN) \; m\geq n \Rightarrow f(m) \neq g(m).
\]
A (maximal) discrete set with respect to this relation is called an \emph{(maximal) eventually different family} (short: \emph{ed family} resp.\ \emph{med family}).
Horowitz and Shelah \cite{medf-borel} showed that there is a Borel such family (see also \cite{on-shelah} for a simpler proof); this was improved by the second author of the present article by constructing a closed (in fact $\Pi^0_1$) such family \cite{on-shelah,medf-bounded}.

\medskip

%In the present paper we deal with the possible size of coanalytic \emph{med families}.
A Borel (or equivalently, analytic) \emph{med} family must always have size $\mathfrak{c}$ (i.e., the size of the continuum, $2^\omega$).
Thus in models where $\mathfrak{c}\geq\aleph_2$, the definability \emph{med} families of cardinality 
strictly smaller than $\mathfrak{c}$ becomes of interest. In fact, the question how simple the definitions of various combinatorial sets of reals can be in models of large continuum has been a central theme in the study of the combinatorial properties of the real line (see \cite{blass} for an introduction and for definitions of the cardinal invariants $\mathfrak{b}$ and $\mathfrak{d}$;  \cite{barto} may also serve as a reference). These include the existence of definable maximal almost disjoint families (\cite{SFLZ, VFSFLZ, JBYK, VFSFYK}), definable maximal families of orthogonal measures (\cite{VFATOM, VFSFATOM, mof, schrittesser-mof-large-continuum}) very recently maximal independent families (\cite{JBVFYK}), and definable maximal cofinitary groups. 

A Borel maximal cofinitary group was constructed by Horowitz and Shelah \cite{borel-mcg}.
The existence of a co-analytic maximal cofinitary group of size $\aleph_1$ in a model of $\mathfrak{b}=\aleph_1<\mathfrak{c}=\mathfrak{d}=\kappa$ is known for arbitrary regular uncountable $\kappa$ (see~\cite{VFDSAT}). The core of the argument is the existence of a co-analytic Cohen indestructible maximal cofinitary group (in the constructible universe $\eL$). While the argument of~\cite{VFDSAT} can be easily modified to produce a co-analytic Cohen indestructible maximal eventually different family, the existence of a co-analytic maximal eventually different family of size $\aleph_1$ in a model of $\mathfrak{d}=\aleph_1<\mathfrak{c}=\kappa$ requires a different approach.  

Such an approach is suggested by a recent analysis of Sacks indestructibility of maximal independent families (see ~\cite{She92},~\cite{VFDM},~\cite{JBVFYK}) showing that in the constructible universe there is a co-analytic Sacks indestructible maximal independent family (see~\cite{JBVFYK}). 
This clearly implies that the existence of a co-analytic maximal independent family of size $\aleph_1$ is consistent with $\mathfrak{d}=\aleph_1<\mathfrak{c}=\kappa$.\footnote{The existence of a Cohen indestructible maximal cofinitary group in $\eL$ with a relatively simple definition stands in strong contrast with the fact that in the Cohen model, there are no projective maximal independent families.}

\medskip

In the present paper, we show that in the constructible universe $\eL$, there is a Sacks indestructible (indestructible under  countable support products and countable support iterations of Sacks forcing) maximal eventually different family which has a $\Pi^1_1$ definition. We thus obtain: 

\begin{thm}\label{t.mainmain} The existence of a co-analytic (in fact lightface $\Pi^1_1$) \emph{med} family of size $\aleph_1$ is consistent with $\mathfrak{d}=\aleph_1<\mathfrak{c}=\kappa$,
where $\kappa$ is an arbitrary cardinal of uncountable cofinality.
\end{thm}

Crucial to our proof of this theorem is the following lemma, which may be of independent interest in descriptive set theory. We will prove this in Section~\ref{s.sacks} in slightly different form, namely as Lemma~\ref{l.main.product}.
\begin{lem}\label{l.dst}
Let $\mathcal F_0$ be a countable set of functions from $\bbN$ to $\bbN$ and suppose $f^* \colon {}^\bbN ({}^{\bbN}2) \to {}^\bbN\bbN$ is a continuous function
 such that
\[
\big(\forall \bar x \in {}^\bbN ({}^{\bbN}2)\big)\big(\forall f\in\mathcal F_0\big) \; f^*(\bar x)\text{ is eventually different from }f
\]
Then there is  $h \in {}^\bbN\bbN$ and a sequence $\langle P_n \mid n\in\omega\rangle$ of perfect subsets of ${}^\bbN 2$ such that
\begin{equation*}
(\forall \bar x \in \prod_{n\in\omega} P_n) \; h\text{ is not eventually different from }f^*(\bar x)
\end{equation*}
but 
$h$ is eventually different from every function in $\mathcal{F}_0$.
\end{lem}

Closely related is the following property of a forcing $\bbP$ (we shall refer to this as ``property $\ned$'', see Definition~\ref{d.ned}): namely, that for every countable ground model set $\mathcal F_0 \subseteq {}^\bbN\bbN$ and every $f\colon\bbN \to \bbN$ added by $\bbP$ which is eventually different from every function in $\mathcal F_0$ there is a function $h$ in the ground model which is not eventually different from $f$ but which is eventually different from every function in $\mathcal F_0$.
We show that property $\ned$ is equivalent to ``not adding an eventually different real'' for ${}^\omega\omega$-bounding forcing (see Proposition~\ref{prop.ned}), which is an iterable property for sufficiently definable such forcings (Proposition~\ref{prop.general-case}).
We also show arbitrary countably supported products of Sacks forcing have this property (Proposition~\ref{prop.general-case}).

\medskip

In a forthcoming paper, we will extend our main result to the existence of a Sacks indestructible co-analytic maximal cofinitary group in $\eL$, thus obtaining the existence of a co-analytic maximal cofinitary group of size $\aleph_1$ in a model of $\mathfrak{d}=\aleph_1<\mathfrak{c}=\kappa$ for an $\kappa$ arbitrary regular uncountable cardinal. Furthermore since uncountable co-analytic sets must have size either $\omega_1$ or $2^\omega$, to obtain other sizes, one must go beyond the co-analytic realm. For a (forcing) construction of a model in which there is a $\Pi^1_2$ \emph{med} family (and in fact a $\Pi^1_2$ maximal cofinitary group) of size $\omega_2$ while $\mathfrak{c}=\aleph_3$ see~\cite{VFSFDSAT}.

\medskip

The paper is structured as follows.
In Section~\ref{s.prelim} we collect some notation and a few basic facts that will be needed throughout.

In Section~\ref{s.sacks} we prove some results about Sacks forcing:
In Lemma~\ref{l.continuous.reading}, we provide a particularly strong version of what is known as \emph{continuous reading of names} which will be convenient for our main result.
We then remind the reader of the proof that Sacks forcing does not add eventually different reals (Lemma~\ref{l.sacks-ev-diff}).
We then prove Lemma~\ref{l.dst} in slightly different form (Lemma~\ref{l.main.product}), i.e., that property $\ned$ holds ``in an absolute manner'' for countable products of Sacks forcing; this together with our particular form of continuous reading of names will be  crucial in the proof of our main result, Theorem~\ref{t.main}.

In Section~\ref{s.main} we show that any $\Sigma^1_2$ \emph{med} family gives rise to a $\Pi^1_1$ such family which is indestructible by the same forcings (Theorem~\ref{t.sigma-1-2-implies-pi-1-1}) and then show how to construct a $\Sigma^1_2$ \emph{med} family which is indestructible by Sacks forcing in $\eL$ (Theorem~\ref{t.main}).

Finally Section~\ref{s.property.ned} we explore property ned (see Definition~\ref{d.ned}) in greater generality, discussing its relation to the property of not adding an eventually different real, its iterability, and show that countably supported iterations and products of Sacks forcing have property ned.

We close with some open questions in Section~\ref{s.open}. 

\medskip

\textit{Acknowledgments:} The authors thank the Austrian Science Fund (FWF) for generous support through grants numbers Y1012-N35 (Vera Fischer) and F29999 (David Schrittesser).
We thank J\"org Brendle for calling to our attention the problem of constructing a co-analytic Sacks indestructible \emph{med} family in $\eL$ during the RIMS meeting in fall 2016.

\section{Notation and preliminaries}\label{s.prelim}

We use both $\omega$ and $\bbN$ for the set of natural numbers. 
We write ${}^A B$ for the set of functions from $A$ to $B$. For an ordinal $\alpha$, 
${}^{<\alpha} B$ denotes $\bigcup_{\xi<\alpha} {}^\xi B$ (the set of sequences from $B$ of length $<\alpha$), and ${}^{\leq\alpha} B$ is just ${}^{<\alpha+1} B$. 
For $s \in {}^{<\alpha} B$ we write $\lh(s)$ for the length of $s$, that is, for $\dom(s)$.
We denote by $\parfin(A,B)$ the set of finite partial functions from $A$ to $B$.
%Of course ${}^\omega 2$ and ${}^\omega\omega$ carry their natural topology as Cantor resp.\ Baire space, and we think of ${}^{<\alpha}({}^\omega 2)$ as equipped with the product topology.

If $S$ is a set equipped with a strict ordering $<_{S}$, ${}^{\leq\omega} S$ is ordered by the \emph{maximolexicographic ordering}, denoted here by $\mathbin{<^{\text{m}}_{\text{lex}}}$ and defined as follows:
Given $s, s' \in {}^{\leq\omega} S$, $s \mathbin{<^{\text{m}}_{\text{lex}}} s'$ if
$\lh(s) < \lh(s')$ or $\lh(s)=\lh(s')$ and for the smallest $k < \lh(s)$ such that $s(k) \neq s'(k)$, it holds that
$s(k) \mathbin{<_S} s'(k)$. 
Here, $<_S$ will itself be the lexicographic order again, or the natural order on $\bbN$. We will not be explicit about which order $<_S$ is meant since this will be clear from the context.

\medskip

Following \cite[2.6]{kechris1995}, we say $f \colon {}^{<\omega} 2 \to {}^{<\omega}\omega$ is \emph{monotone}
if whenever $s,s' \in {}^{<\omega}2$ and $s \subseteq s'$, then $f(s) \subseteq f(s')$.
We say such $f$ is \emph{proper}  if
for any $x \in {}^{\omega} 2$,
\[
\lh(f(x\res n)) \stackrel{n\to\infty}{\longrightarrow} \infty
\]
Given such $f$, we define a continuous function $f^*\colon {}^{\omega}2\to{}^{\omega}\omega$ by
\[
f^*(x) = \bigcup_{n\in\omega} f(x\res n)
\]
We also say $f$ \emph{is a code} for $f^*$.

Any continuous function $h\colon {}^{\omega}2 \to {}^{\omega}\omega$---or more generally, between effective Polish spaces---arises from a code, and being a code for a continuous function (i.e., being monotone and proper) is absolute between well-founded models of (a small fragment of) set theory.
Analogously, we will talk about
monotone and proper functions
\[
f\colon {}^{<\omega} ({}^{<\omega} 2) \to {}^{<\omega}\omega.
\]
in which case $f^*$ is a total function ${}^\omega({}^\omega 2) \to {}^\omega\omega$.
All the above holds, the necessary changes having been made; we leave details to the reader.

\medskip

Recall that Sacks forcing $\bbS$ is the partial order consisting of perfect subtrees of ${}^{<\omega} 2$, ordered by reverse inclusion. Given $p \in \bbS$, a \emph{splitting node in $p$} is a node $t\in p$ such that $\{t \conc 0, t\conc 1\} \subseteq p$, and $p_n$ denotes the set of $n$th splitting nodes, that is, the set of splitting nodes $t$ in $p$ such that among the proper initial segments of $t$ there are exactly $n$ splitting nodes; $p_{\leq n}$ denotes $\bigcup_{k\leq n} p_k$.
Given $t\in p$, $p_t=\{s\in p\mid s\subseteq t\vee t\subseteq s\}$.
We write
$[p] = \{x\in{}^\omega 2\mid(\forall n\in\omega)\; x\res n\in p\}$ 
for the set of branches of $p$.
When $\bbP $ is a product of Sacks forcing and $\tilde p \in \bbP$, we shall use the same notation and write
\[
[\tilde p]=\big\{\bar x\in {}^{\supp(\tilde p)}({}^\omega 2) \mid \big(\forall \xi\in\supp(\tilde p)\big)\; \bar x(\xi)\in [\tilde p(\xi)]\big\}.
\]
Below in Definition~\ref{d.main}\eqref{d.iterated-branch} we will introduce similar a notation for iterated Sacks conditions.

Given an iteration $\bbP$ of Sacks forcing of length $\lambda$, a condition $\bar p\in \bbP$ and $\bar t \in \parfin(\supp(\bar p),{}^{<\omega} 2)$,
we define by induction a sequence $\bar p_{\bar t}$ of length $\lambda$ as follows: Clearly, $\bar p_{\bar t} \res \emptyset$ is the empty sequence.
Given $\sigma \in \lambda $, suppose $\bar p_{\bar t} \res \sigma$ is already defined. 
If $\sigma \in\dom(\bar t)$ let $\bar p_{\bar t}(\sigma)$ be a
$\bbP \res \sigma$-name such that
\[
\forces_{\bbP\res\sigma} \bar p_{\bar t}(\sigma)=\bar p(\sigma)_{\bar t(\sigma)}.
\]
Otherwise, let $\bar p_{\bar t}(\sigma) = \bar p(\sigma)$.

Given $\bar p\in \bbP$ and $\bar t \in \parfin(\supp(\bar p),{}^{<\omega} 2)$, we say \emph{$\bar p$ accepts $\bar t$} 
to mean that $\bar p_{\bar t} \in \bbP$, i.e., the above definition yields an iterated Sacks conditions---equivalently, for each $\sigma \in \dom(\bar t)$,
\[
\bar p_{\bar t} \res\sigma\forces_{\bbP \res\sigma} \bar t(\sigma) \in \bar p(\sigma).
\]
When $\bbP$ is a product of Sacks forcing and $\tilde p \in \bbP$, $\tilde p_{\bar t}$ and the relation ``accepts'' is defined analogously.
We usually denote by $\bar s_{\dot G}$ a name for the generic sequence of Sacks reals in a product or iteration of Sacks forcing.

\medskip

Recall that  a forcing $\bbP$ is \emph{${}^\omega\omega$-bounding} if any only if
\begin{equation*}
\forces_{\bbP} (\forall f\in{}^\omega\omega)(\exists h\in{}^\omega\omega\cap V)(\forall n\in\omega)\;f(n) \leq h(n).
\end{equation*}
We say $\bbP$ \emph{does not add an eventually different real} to mean
\begin{equation*}
\forces_{\bbP} (\forall f\in{}^\omega\omega)(\exists h\in{}^\omega\omega\cap V)\;\text{$f$ is not eventually different from $h$.}
\end{equation*}
We sometimes---but not always---decorate names in the forcing language with checks and dots, with the goal of helping the reader. 
We refer to \cite{jech-set,barto,kechris1995,moschovakis,kanamori} for further notation and notions left undefined here.

\section{Iterated Sacks forcing and eventually different reals}\label{s.sacks}

It will be convenient for our main result to have a particular form of ``continuous reading of names'' for iterated Sacks forcing which is more easily stated after we introduce some terminology (some of which is taken from~\cite{schrittesser-mof-large-continuum}).
\begin{defi}\label{d.main} Let $\bbP$ be a countable support iteration of Sacks forcing of any length and suppose $\bar p, \bar q \in \bbP$.
\begin{enumerate}[label=(\alph*), ref=\alph*]
\item Given $p\in\bbS$, let $p^*_n$ denote the set of immediate successors of $n$th splitting nodes, i.e., $p^*_n=\{s\in p\mid s \res (\lh (p)-1) \in p_n\}$.
\item A \emph{standard enumeration of $\supp(\bar p)$} is a sequence
\[
\Sigma = \langle \sigma_l \mid  l < \alpha\rangle
\]
such that $\alpha \leq \omega$,
$\{ \sigma_l \mid  l < \alpha \} = \{0\}\cup\supp(\bar p)$, and $\sigma_0 = 0$.

\item
Given $n \in\omega$, a finite subset $S$ of $\{0\}\cup\supp(\bar p)$, and $n \in \omega$ we write
$\bar q\leq^{S}_n \bar p$ exactly if $\bar q\leq \bar p$ and for every $\sigma\in S$,
\[
\bar q \res \sigma \forces \bar q(\sigma)_n = \bar p(\sigma)_n.
\]
If $S=\{\sigma_0, \hdots, \sigma_k\}$ and $\Sigma$ is a sequence whose first $k+1$ values are $\sigma_0, \hdots, \sigma_k$, we also write $\leq^{\sigma_0, \hdots, \sigma_k}_n$ or $\leq^{\Sigma\res k+1}_n$  
for $\leq^{S}_n$. Note that $\leq^\emptyset_n$ is just $\leq$.

\item\label{d.iterated-branch}
Let $[\bar p]$ be some name such that
\[
\forces_{\bbP} [\bar p] = \big\{\bar x\in {}^{\{0\}\cup\supp(\bar p)}({}^{\omega} 2) \mid \big(\forall \sigma \in\{0\}\cup\supp(\bar p)\big) \; \bar x(\sigma) \in\bar [p(\sigma)]\big\}.
\]
\item Let $\Sigma=\langle \sigma_k \mid k\in \alpha\rangle$ be standard enumeration of $\supp(\bar q)$.
For each $k\in\alpha$ let $\dot e^{\bar q, \Sigma}_k$ be a $\bbP\res\sigma_k$-name such that
\[
\forces_{\bbP\res \sigma_{k}} \dot e^{\bar q, \Sigma}_k \colon  [q(\sigma_{k})] \to {}^{\omega}2\text{ preserves the strict lexicographic ordering and is onto.}
\]
Note that $\bbP \res \sigma_0 = \{\emptyset\}$, i.e., the trivial forcing and we treat $\dot e^{\bar q, \Sigma}_0$ as a name only to simplify notation.

Moreover, let $\dot e^{\bar q, \Sigma}$ be a name for the map $[\bar q] \to {}^\alpha({}^\omega 2)$ such that
\[
\forces_{\bbP}  \dot e^{\bar q, \Sigma}(\bar x)= \langle \dot e^{\bar q, \Sigma}_k( \bar x (\sigma_k)) \mid k \in \alpha \rangle.
\]
We drop the superscript $\Sigma$ as well as $\bar q$ when they can be inferred from the context.
\end{enumerate}
We shall use the same notation for products; most of the definitions above then simplify slightly. Details are left to the reader.
\end{defi}
The next lemma is a version of continuous reading of names for iterations and products of Sacks forcing.
\begin{lem}\label{l.continuous.reading}
Let $\bbP$ be the countable support iteration or product of Sacks forcing of length $\lambda$.
Suppose $\bar p\in\bbP$ and $\dot f$ is a $\bbP$-name such that $\bar p\forces \dot f \in {}^\omega\omega$.
Then we can find $\bar q\in\bbP$ together with a standard enumeration $\Sigma=\langle \sigma_k \mid k \in \alpha\rangle$ of $\supp(\bar q)$  and a code $f\colon {}^{<\alpha}({}^{<\omega}2)\to{}^{<\omega}\omega$ for a continuous map
\[
f^* \colon {}^{\alpha}({}^{\omega} 2) \to {}^\omega\omega
\]
such that (dropping superscripts from $\dot e^{\bar q, \Sigma}$ and letting $\bar s_{\dot G}$ be a name for  the generic sequence of Sacks reals)
\[
\bar q \forces_{\bbP} \dot f = (f^* \circ \dot e)(\bar s_{\dot G} \res \supp(\bar q)).
\]
\end{lem}
\begin{proof}
The proof for products is a straightforward simplification of the proof for iterations, so let us suppose $\bbP$ is an iteration as in the lemma. 
Let $\bar p \in \bbP$ and let $\dot f$ be a $\bbP$-name such that $\bar p\forces \dot f\in {}^\omega\omega$.
To obtain $f$ and $\bar q$, we will build a fusion sequence $\bar p^0 \geq \bar p^1 \geq \bar p^2 \hdots$ with $\bar p^0=\bar p$ and whose greatest lower bound will be $\bar q$, as follows:
Fix a standard enumeration $\Sigma^0=\langle \sigma^0_k \mid  k\in\alpha_0\rangle$ of $\supp(\bar p^0)$.
For every further step $n >0$ in the construction of the fusion sequence, after having obtained $\bar p^n$  we shall also fix an enumeration $\Sigma^n=\langle \sigma^n_k \mid  k\in\alpha_n\rangle$ of $\supp(\bar p^{n})\setminus \supp(\bar p^{n-1})$ . Note that the choice of enumeration at each step is essentially arbitrary; but let us assume $0\in\supp(\bar p^0)$ to avoid trivialities.
Also note that it is entirely possible that for some $n>0$, $\Sigma^n$ is the empty sequence.\footnote{\label{footnote}Of course, we could distinguish cases as follows: If $\lambda < \omega_1$ we may assume $\supp(p^0) = \lambda$ and $\alpha_n=\emptyset$ for $n>0$; if $\lambda \geq \omega_1$ we may assume $\alpha_n = \omega$ for each $n\in\omega$.}

At the end of the construction we obtain a standard enumeration $\Sigma=\langle\sigma_m\mid m<\alpha\rangle$ of $\supp(\bar q)$ as follows:
Let 
\[
\diagbij \colon \bigcup_{n\in\omega} \big(\{n\} \times \alpha_n\big) \to \omega
\] 
be any bijective function such that $\diagbij(0)=(0,0)$ and for each $m\in\omega$,
\begin{equation}\label{e.d}
\diagbij^{-1}[\{0, \hdots, m\}]\subseteq \bigcup_{n\leq m} \big(\{n\} \times \alpha_n\big)
\end{equation}
Then we shall let $\Sigma = \langle \sigma_m \mid  m\in\alpha\rangle$ be defined by
\[
\sigma_{\diagbij(n,k)} = \sigma^n_k.
\]
For concreteness, assume that for each $n\in\omega$, $\alpha_n = \omega$. Then we can let
\[
\diagbij(n,k)= \frac{(n+k+1)(n+k)}{2}+n.
\]
In other words, $\Sigma$ will enumerate $\{ \sigma^n_k \mid  n,k\in \omega \}$ by the well-known \emph{diagonal counting procedure}:

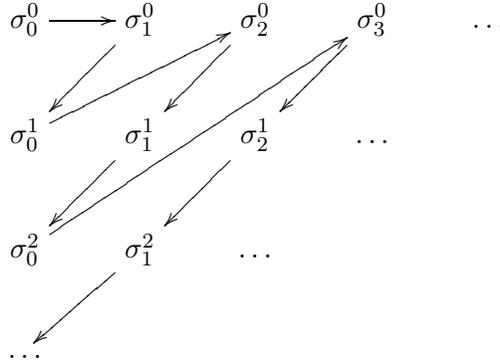
\[
%\xymatrix@C=0.6cm{
%\xymatrix@C=0.6cm@R=1cm{
\xymatrix{
   {\sigma^0_0} \ar@{->}[r]&  {\sigma^0_1} \ar@{->}[ld]  & {\sigma^0_2} \ar@{->}[dl]& {\sigma^0_3} \ar@{->}[dl] &\hdots \\
  \sigma^1_0 \ar@{->}[rru]  &  \sigma^1_1\ar@{->}[dl]  & {\sigma^1_2} \ar@{->}[dl] & \hdots  &\\
  \sigma^2_0 \ar@{->}[rrruu] &  \sigma^2_1 \ar@{->}[dl] &\hdots &  &\\
  \hdots&&&&\\
 }
\]
%Equation~\ref{e.d} holds because at step $n$ we have determined $\sigma_{d(n',k')}$ for all $k'\in \alpha_{n'}$ and all $n' < n\cap\alpha$, so the first index which has not yet been assigned a value in $\Sigma$ is $d(n,0)$ and $n < d(n,0)$ since $n\geq 1$, thus determining $\langle \sigma_0, \hdots, \sigma_{n}\rangle$.

\medskip

In case some (or all) $\alpha_n$ are finite, we leave it to the reader to vary this construction to find an appropriate map $\diagbij$---but note from the footnote on p.~\pageref{footnote} that in this case, one can simply assume $\Sigma=\Sigma^0$ and that $\supp(\bar p^n)  = \lambda$ for each $n\in\omega$.

\medskip

Our construction will be set up so as to guarantee that
\[
\bar p^0 \geq^{\sigma_0}_0 \bar p^1 \geq^{\Sigma\res 2}_1 \geq \hdots \bar p^{n-1} \geq^{\Sigma\res n}_{n-1} \bar p^{n}\hdots
\]
At the same time we shall define $f$ and auxiliary functions which approximate the inverse of $\dot e^{\bar q,\Sigma}$: For each $n>0$ and $k$ such that $0<k\leq n\cap \alpha$, we define
\[
d^k_n\colon {}^{k}({}^n 2) \to {}^{\{\sigma_l \mid l< k\}}({}^{<\omega}2)
\]
so that the following hold:
\begin{enumerate}[label=(\Alph*), ref=(\Alph*)]
\item\label{r.aux-first}
For each $\bar t \in {}^{k}({}^n 2)$, $\bar p^n$ accepts $d^k_n(\bar t)$, that is, $\bar p^n_{d^k_n(\bar t)}$ is an iterated Sacks condition (see Section~\ref{s.prelim} for the meaning of ``accepts''). 

\item\label{r.d-1-1}
The map defined for $t \in {}^n 2$ by
\[
t \mapsto d^{1}_n(\langle t\rangle)(\sigma_{0}),
\]
where $\langle t\rangle$ denotes the sequence of length $1$ whose only element is $t$,
is the unique bijection
\[
{}^n2 \to \bar p^n(0)^*_{n-1}
\]
which preserves the lexicographic ordering.

\item\label{r.d-n-k-coheres} 
If $k'$ is such that $k < k' \leq n\cap \alpha$, the map $d^{k'}_n$ extends $d^k_n$ in the sense that
for any $\bar t \in {}^{k'}({}^n 2)$,
\[
d^{k'}_n(\bar t) \res \sigma_k = d^k_n(\bar t \res k).
\]

\item\label{r.e-k-n}\label{r.aux-last} 
When $\bar t \in  {}^{k}({}^n 2)$,
the condition $(\bar p^n\res\sigma_{k})_{d^{k}_n(\bar t)}$ forces that in the extension by $\bbP \res \sigma_{k}$,
the map defined for $t \in {}^n 2$ by
\[
t \mapsto d^{k+1}_n(\bar t \conc t)(\sigma_{k})
\]
is the unique bijection
\[
{}^n2 \to \bar p^n(\sigma_{k})^*_{n-1}
\]
which preserves the lexicographic ordering on ${}^{<\omega} 2$.

\item\label{r.f} For each $\bar t\in {}^{n\cap\alpha}({}^n 2)$, 
$
(\bar p^n)_{d^n_n(\bar t)} \forces_{\bbP} \dot f \res n = f(\bar t)$.
\end{enumerate}

\medskip

 \begin{rem}~\label{r.remark}
 \begin{enumerate}
 \item Note that by Requirement~\ref{r.e-k-n} in particular $(\bar p^n \res \sigma_k)_{d^k_n(\bar t)}$ decides $\bar p(\sigma_k)^*_{n-1}$, and in fact by induction decides $\bar p(\sigma_k)_{\leq n-1}$. 
 \item The reader should recognize that due to the fact that the $\Sigma$ cannot in general enumerate the support in its natural order, often $d^k_n(\bar t)$ will not depend on some of the components of its input $\bar t$---precisely, it only depends on those $\bar t(j)$ where both $j < k$ and $\sigma_j < \sigma_k$. Nevertheless, it is notationally much easier to allow `dummy input' in the other components.
 \item Since ${}^{0}({}^n 2) = \{ \emptyset \}$, i.e., contains only the empty sequence, \ref{r.d-1-1} can be interpreted as the case  of \ref{r.e-k-n} where we allow $k=0$ (where necessarily $\bar t = \emptyset$, and $\bbP\res\sigma_0$ is the trivial forcing).
 \item By these requirements, in particular Requirement \ref{r.f}, clearly $f$ will be monotone and proper and thus code a continuous function
 $f^*\colon {}^{<\alpha}({}^\omega 2)\to {}^\omega\omega$.

\end{enumerate}
\end{rem}

\medskip
 We now give the details of the construction.
 The definition of $d^k_n$ is by induction on $n$, with several finite sub-inductions, one of them on $k\leq n$, at each step. Simultaneously, we define $f$.

 Let $\bar p^0 \leq \bar p$ be any condition such that $0\in\supp(\bar p^0)$ to avoid trivialities. 
 We first define $f$ restricted to ${}^0({}^0 2)$ (i.e., to $\{\emptyset\}$, the set containing the empty function) 
 by letting $f(\emptyset)=\emptyset$.

\medskip

For arbitrary $n\geq 1$, suppose by induction we have already defined $\bar p^{n-1}$, and $f$ restricted to
\[
 \bigcup_{n' < n} {}^{n'\cap\alpha}({}^{n'} 2).
\]
We first define $d^k_{n}$ by induction on $k\in\{1, \hdots, n\cap\alpha\}$.
The definition will be so that Requirements~\ref{r.aux-first}--\ref{r.d-n-k-coheres} above are satisfied.
At the same time, we build a sequence $\bar p^{n,1} \geq^{\Sigma \res n}_{n-1}  \bar p^{n,2} \hdots \geq^{\Sigma \res n}_{n-1} \bar p^{n,n\cap\alpha}$, starting with $\bar p^{n-1}$. 
It is crucial here that by Equation \eqref{e.d} we have already determined the first $n\cap\alpha$ elements of $\Sigma$, i.e., $\langle \sigma_l \mid l < n\cap \alpha\rangle$.

So let $\bar p^{n,1} = \bar p^{n-1}$.
Now define $d^1_{n}$ as follows.
For the moment denote by $d^*_{n}$ the unique map  
\[
d^*_{n}\colon{}^{n}2\to \bar p^{n,1}(0)^*_{n-1}
\] 
which preserves the lexicographic ordering. 
Recalling again that $\sigma_0=0$, for each $\bar t \in {}^{\{0\}}({}^{n} 2)$,
define $d^1_{n}(\bar t)$ be the unique $\bar{t}' \in{}^{\{\sigma_0\}}(\bar p^{n,1}(\sigma_0)^*_{n-1})$ such that
\[
\bar{t}' (0) = d^*_n(\bar t(0)).
\]
Clearly, \ref{r.d-1-1} is satisfied.

Now assume $1 \leq k < n\cap \alpha$ and suppose by induction that $d^k_{n}$ and $\bar p^{n,k}$ are already defined.
We shall presently define $d^{k+1}_{n}$.
Also, in finitely many steps we reach $\bar p^{n,k+1}\leq^{\Sigma \res n}_{n-1}\bar p^{n,k}$
such that for each $\bar t \in {}^{k}({}^{n}2)$ we can find $T(\bar t)\subseteq {}^{<\omega}2$ 
so that
\begin{equation}\label{e.fix-splitting-level}
(\bar p^{n,k+1}\res \sigma_{k})_{d^k_{n}(\bar t)}\forces_{\bbP\res\sigma_{k}} \bar p^{n,k+1}(\sigma_{k})^*_{n-1} = \check T(\bar t).
\end{equation}
Now let $d^{k+1}_{n}$ be the unique map satisfying \ref{r.d-n-k-coheres} such that for each $\bar t \in {}^{k}({}^{n}2)$ and $t \in {}^{n}2$,
\[
t \mapsto d^{k+1}_{n}(\bar t\conc t)(\sigma_k)
\]
is the unique map
\[
{}^{n} 2 \to T(\bar t)
\]
which preserves lexicographic order. 
Clearly \ref{r.aux-first} and \ref{r.e-k-n} are also satisfied by construction.

\medskip

For each $\bar t \in {}^{n\cap\alpha}({}^n 2)$ we shall now define $f(\bar t)$ and a condition $p^{\bar t}$ by induction on the lexicographic order on
${}^{n\cap\alpha}({}^n 2)$, starting with the outcome of the previous finite sub-induction, i.e., $\bar p^{n,n\cap\alpha}$.
We ensure that $\langle \bar p^{\bar t}\mid\bar t \in {}^{n\cap\alpha}({}^n 2)\rangle$ is a $\leq^{\Sigma\res n}_{n-1}$-descending sequence in $\bar t$ with respect to the lexicographic ordering.

\medskip

Suppose now that we have already defined $f$ on the all elements ${}^{n\cap\alpha}({}^n 2)$ of which come before $\bar t$ in lexicographic order.
If $\bar t$ is lexicographically minimal in ${}^{n\cap\alpha}({}^n 2)$, let $\bar p^* = \bar p^{n,n\cap\alpha}$, otherwise let $\bar p^* = \bar p^{\bar t^*}$ where $\bar t^*$ is the lexicographic predecessor of $\bar t$.
We may find $\bar p^{\bar t} \leq^{\Sigma\res n}_{n-1} p^{\bar t^*}$
and $f(\bar t) \in {}^n \omega$ such that
\begin{equation}\label{e.f}
(\bar p^{\bar t})_{d^n_n(\bar t)} \forces_{\bbP} \dot f \res n = \check f(\bar t).
\end{equation}

\medskip

Finally, we let $\bar p^n$ be  the condition reached after having done this for every $\bar t \in {}^{n\cap\alpha}({}^n 2)$---i.e., $\bar p^n= \bar p^{\bar t_*}$ where $\bar t_*$ is maximal in ${}^{n\cap\alpha}({}^n 2)$ with respect to the lexicographic ordering.
Since $\bar p^n \leq \bar p^{\bar t}$ for every $\bar t \in {}^{n\cap\alpha}({}^n2)$,  Equation \eqref{e.f} shows that Requirement \ref{r.f} is satisfied. 

\medskip

Let $\bar q$ be the greatest lower bound of the sequence $\langle \bar p^n \mid n\in\omega\rangle$.
It is clear that $f$ is monotone and proper and thus codes a function on ${}^\alpha({}^\omega 2)$.
We verify that
\[
\bar q \forces_{\bbP} \dot f = (f^* \circ \dot e^{\bar q,\Sigma})(\bar s_{\dot G}\res \supp(\bar q)).
\]
First, letting
\[
\codeforinvofe = \bigcup_{n\in\omega\setminus\{0\}} d^{n\cap\alpha}_n
\]
we obtain a map
\[
\codeforinvofe\colon \bigcup_{n\in\omega} {}^{n\cap\alpha}({}^n 2) \to \parfin(\supp(\bar q), {}^{<\omega} 2).
\]
Note that $\bar q \leq^{\Sigma \res n}_{n-1} \bar p^{n,k}$ for every $n \in \omega\setminus\{0\}$ and $k \in \{1, \hdots, n\cap\alpha\}$, so that Equation \eqref{e.fix-splitting-level} still holds if we replace $\bar p^n_{k}$ by $\bar q$ everywhere.

Therefore if $\bar t \in {}^{n\cap\alpha}({}^n 2)$ for $n>0$, $\bar q$ accepts $\codeforinvofe(\bar t)$, $\dom(\codeforinvofe(\bar t))=\{\sigma_0, \hdots, \sigma_{(n\cap\alpha)-1}\}$ and for each $\sigma \in \dom(\codeforinvofe(\bar t))$
\[
\bar q_{\codeforinvofe(\bar t)}\forces_{\bbP} \codeforinvofe(\bar t)(\sigma) \in \bar q(\sigma)^*_{n-1}
\]
and in fact for each $k$ such that $1\leq k < n\cap\alpha$
\[
(\bar q \res \sigma_{k})_{\codeforinvofe(\bar t)\res k}\forces_{\bbP\res\sigma_{k}} \bar q(\sigma_{k})^*_{n-1} = T(\bar t\res k)
\]
and $(\bar q\res\sigma_{k})_{\codeforinvofe(\bar t)\res k}$  also forces that the map $t \mapsto \bar t\res k \conc t$ is the unique lexicographically order preserving map
from $\bar q(\sigma_{k})_n$ to ${}^n 2$.
Therefore $\codeforinvofe$ must be a code for the inverse of $\dot e^{\bar q, \Sigma}$.

Moreover, for any $\bar t \in \bigcup_{n\in\omega} {}^{n\cap\alpha}({}^n 2)$ it holds by construction that
\[
(\bar q)_{\codeforinvofe(\bar t)} \forces_{\bbP} f(\bar t) \subseteq \dot f
\]
and so
\[
\bar q \forces_{\bbP} \dot f = \bigcup \big\{ f(\bar t) \mid \bar t \in \bigcup_{n\in\omega} {}^n({}^n 2) \wedge \big(\forall \sigma\in\dom(\codeforinvofe(\bar t))\big)\; \codeforinvofe(\bar t)(\sigma) \subseteq \bar s_{\dot G} (\sigma)\big\}
\]
and hence
$\bar q \forces_{\bbP} \dot f = f^* \circ (\codeforinvofe^*)^{-1} (\bar s_{\dot G})$. 
Since we have shown that $\bar q\forces_{\bbP}\codeforinvofe^* = (\dot e^{\bar q, \Sigma})^{-1}$, this finishes the proof.
\renewcommand{\qedsymbol}{{\tiny  Lemma~\ref{l.continuous.reading}~}$\Box$}
\end{proof}

\medskip

It is crucial to our construction that Sacks forcing does not add a function from $\bbN$ to $\bbN$ which is eventually different from all such functions in the ground model.
For the readers convenience and to provide a blueprint for the more complicated construction in Lemma~\ref{l.main.product} below,
we give a simple proof. %The proof uses that Sacks forcing has ``continuous reading of names,'' that is, for any Sacks name $\dot x$ such that $\forces_{\bbS}\dot x \in {}^\omega\omega$, there is $p\in\omega$ and a code $f\colon p \to {}^{<\omega}\omega$ such that $p\forces_{\Sacks} \dot x = f^*(s_{\dot G}$
\begin{lem}\label{l.sacks-ev-diff}
%Sacks forcing does not add an element of ${}^\omega\omega$ which is eventually different from every element of ${}^\omega\omega$ in the ground model.
%In other words, suppose $p\in\bbS$ and $\dot f$ is a Sacks name such that $p \forces_{\bbS} \dot f\in {}^\omega\omega$ .
Suppose $p\in\bbS$ and $\dot f$ is a Sacks name such that $p \forces_{\bbS} \dot f\in {}^\omega\omega$ .
There is $h \in {}^\omega\omega$ and a condition $q \in \bbS$ stronger than $p$ such that
$q \forces_{\bbS}$``$\check h$ and $\dot f$ are not eventually different.''
\end{lem}
\begin{proof}%[Proof of Lemma~\ref{l.sacks-ev-diff}]
Fix a bijection $\# \colon {}^{<\omega} 2 \to \omega$.
Suppose we are given $p$ and $\dot f$ as in the Lemma. We may assume (by strengthening $p$ if necessary) that we can find a monotone and proper function (as discussed in Section~\ref{s.prelim}) $f \colon p \to {}^{<\omega}\omega$ such that (denoting by $s_{\dot G}$ the Sacks generic real) 
\[
p \forces_{\bbS} f^*(s_{\dot G}) = \dot f.
\]
We inductively build a map $c \mapsto t_c$, from ${}^{<\omega}2$ to $p$ and simultaneously define $h$ as follows.
The induction is over the length of $c \in {}^{<\omega} 2$.

Let $t_{\emptyset}$ be some splitting node in $p$ such that $\#(\emptyset) \in \dom(f(t_\emptyset))$.
Also, define $h(\#(\emptyset))$ to be $f(t_\emptyset)(\#(\emptyset))$.

For the inductive step, suppose that we have already defined $t_{c}$.
For each $i\in\{ 0,1\}$, let $t_{c \conc i}$ be some splitting node in $p$ properly extending $t_c$ which is long enough so that
$\#(c\conc i) \in \dom(f(t_{c \conc i})$)).
Also, define $h(\#(c \conc i))$ to be $f(t_{c\conc i})(\#(c \conc i))$.

Finally, letting $q = \{t_c \mid c \in {}^{<\omega} 2\}$, we obviously obtain a Sacks condition below $p$.

\medskip

It remains to show $q$ and $h$ accomplish what is claimed.
Towards a contradiction, we may suppose we have a condition $q' \in \bbS$ stronger than $q$ and $m\in\omega$ such that
\begin{equation}\label{e.catch-name}
q' \forces \forall n \geq m \; f^*(s_{\dot G})(n) \neq \check h(n)
\end{equation}
Since $q' \leq q$, we can find $c \in {}^{<\omega} 2$ so that
$t_c \in q'$ and $\#(c) \geq m$.
But as $(q')_{t_c} \forces \dot f(\#(c)) = \check f(t_c)(\#(c))$ and $f(t_c)(\#(c))=h(\#(c)$, this contradicts Equation~\eqref{e.catch-name}.
\end{proof}

\medskip

For the proof of our main theorem, we need a technical strengthening of the fact that Sacks forcing does not add an eventually different real.
Firstly, the function $h$ which `catches' $\dot f$ is also required to avoid each function in a fixed countable family $\mathcal F_0$ of functions in the ground model. 
Note that the set of $h$ with this property is meager. 
We will see in Section~\ref{s.property.ned} that  for ${}^\omega\omega$-bounding forcings, not adding an eventually different real is in fact equivalent to this (apparently stronger) property.

Secondly, we need a more ``absolute'' version of this property, just as in Lemma~\ref{l.dst}. This is made necessary by the fact that to prove  Theorem~\ref{t.main} we must work within $\eL_{\omega_1}$ and anticipate functions added by an uncountable forcing while working entirely with countable fragments.

\medskip

For the following lemma recall that for $\tilde p \in \bbS^\alpha$, $[\tilde p]$ denotes $\{\bar x \in {}^\alpha({}^\omega 2) \mid (\forall i \in \alpha)\;  \bar x(i) \in [\tilde p(i)]\}$.
The reader will not find it hard to see that essentially the same proof yields Lemma~\ref{l.dst}.

\begin{lem}\label{l.main.product}
Let $\mathcal F_0 \subseteq {}^\omega \omega$ be countable, let $\alpha \leq\omega$, and suppose $f \colon {}^{<\alpha}({}^{<\omega} 2) \to {}^{<\omega} \omega$ is a code for a continuous function
$f^* \colon {}^\alpha({}^\omega 2) \to {}^\omega\omega$ such that (denoting by $\bar s_{\dot G}$ the generic sequence of Sacks reals for the fully supported product $\bbS^\alpha$)
\[
\forces_{\bbS^\alpha} f^*(\bar s_{\dot G}) \text{ is eventually different from every function in $\mathcal F_0$.}
\]
There is  $h \in {}^\omega\omega$ and $\tilde p\in\bbS^\alpha$ such that
\begin{equation}\label{e.eventually-diff}
(\forall \bar x \in [\tilde p]) \; h\text{ is not eventually different from }f^*(\bar x)
\end{equation}
and
$h$ is eventually different from every function in $\mathcal{F}_0$.
\end{lem}

\begin{proof}%[Proof of Lemma~\ref{l.iterated-sacks-no-ed}]
Let $\mathcal F_0$ and $f \colon {}^{<\alpha}({}^{<\omega} 2) \to {}^{<\omega} \omega$ as above be given.
Let $\mathcal F_0$ be enumerated as $\langle f_k\mid k\in\omega\rangle$. 
To obtain $\tilde p \in \bbS^\alpha$, we will build a fusion sequence $\tilde p^0  \geq^\emptyset_0 \tilde p^1 \geq^{1\cap\alpha}_1 \tilde p^2 \geq^{2\cap\alpha}_2 p^3 \hdots$ (note here once more $\leq^\emptyset_0$ is just $\leq$) with ${\tilde p}^0= {}^{\alpha}({}^{<\omega}2)$ whose greatest lower bound will be $\tilde p$.
At the same time we define a partial function $h_0$ from $\bbN$ to $\bbN$ and injective functions
 \begin{align*}
 \overline{\#}\colon& \bigcup_{n\in\omega\setminus\{0\}} \big(\tilde p^{n-1}\res (n\cap\alpha)\big)_n  \to \bbN\\
 \intertext{and}
 d\colon& \bigcup_{n\in\omega\setminus\{0\}} \big(\tilde p^{n-1}\res (n\cap\alpha)\big)_n \to \bigcup_{n\in\omega} \big(\tilde p\res (n\cap\alpha)\big)_n.
 \end{align*}
For arbitrary $n\geq 1$, suppose we have defined $\tilde p^{n-1}$ and $h_0 \circ \overline{\#}$ as well as $d$ restricted to
\[
 \bigcup_{0< n' < n} \big(\tilde p^{n'-1}\res (n'\cap\alpha)\big)_{n'}.
\]
Letting 
\[
T =  \big(\tilde p^{n-1}\res (n\cap\alpha)\big)_n.
\]
we shall now for each $\bar t \in T$ define $\overline{\#}(\bar t)$, $h_0(\overline{\#}(\bar t))$, $d(\bar t)$ and a condition $\tilde p^{\bar t}$ by induction on the lexicographic order.
We ensure that $\langle \tilde p^{\bar t}\mid\bar t \in T \rangle$ is a $\leq^{(n-1) \cap \alpha}_{n-1}$-descending sequence in $\bar t$ with respect to the lexicographic ordering.

\medskip

We begin by letting $\tilde p^{\bar t} = \tilde p^{n-1}$ when $\bar t$ is the first element of $T$.
Suppose now that we have already defined $\overline{\#}$, $h\circ \overline{\#}$, and $d$ on the all elements $T$ of which come before $\bar t$ in lexicographic order, and have reached $\tilde p^{\bar t^*}$ where $\bar t^*$ is the lexicographic predecessor of $\bar t$.
We may find $\bar u \in {}^{<\alpha}({}^{<\omega} 2)$ such that  $\bar u$ extends $\bar t$ 
and natural numbers $m, m'$ so that
\begin{enumerate}[label=(\roman*)]
\item $f(\bar u)(m)=m'$,
\item $m \notin \{ \overline{\#}(\bar t') \mid \bar t' \mathbin{<^{\text{m}}_{\text{lex}}} \bar t, \bar t' \in \dom(\overline{\#})\}$,
\item\label{i.diff.from.F_0} for each $j \leq m$, $m' \neq f_j(m)$,\label{r.e-d}
\item\label{i.splits} For each $k\in \dom(\bar u)$, $\bar u(k)$ is a splitting node of $\bar p^{\bar t^*}(k)$. In particular, $p^{\bar t^*}$ accepts $\bar u$.
\end{enumerate}
Item~\ref{i.diff.from.F_0} is possible because $(f^*)^{\Vee[\dot G]}(\bar s_{\dot G})$ is forced to be eventually different from $f_j$ for each $j\in\omega$.
Let $\tilde p^{\bar t}$ be the unique condition such that for each $\bar s \in T \setminus \{ \bar t \}$,
$(\tilde p^{\bar t})_{\bar s} = (\tilde p^{\bar t^*})_{\bar s}$ and
\[
(\tilde p^{\bar t})_{\bar t} = (\tilde p^{\bar t^*})_{\bar u}
\]
Set $\overline{\#}(\bar t)=m$, $h_0(\overline{\#}(\bar t))=m'$, and $d(\bar t)=\bar u$.
Finally, let $\tilde p^n$ be $\tilde p^{\bar t_{\text{max}}}$ where $\bar t_{\text{max}}$ is the lexicographically maximal element of $T$.
Note for later use that if $\bar t \in (\tilde p^{n-1} \res n)_n$ then $d(\bar t)$ is accepted by $\tilde p^n$ and that
\begin{equation}\label{e.h.catches.f}
f(d(\bar t))(\overline{\#}(\bar t)) = h_0(\overline{\#}(\bar t)).
\end{equation}
Since for each $n \in \omega$, $\tilde p^{n+1} \leq^{n\cap\alpha}_{n} \tilde p^n$, the sequence $\langle \tilde p^n \mid n\in \omega\rangle$ has a greatest lower bound $\tilde p$.
For the same reason and because of Item~\ref{i.splits},
\begin{equation}\label{e.front}
\ran(d)  = \bigcup_{n\in\omega}\big(\tilde p \res (n\cap\alpha)\big)_n
\end{equation}
It is then clear by Requirement~\ref{r.e-d} above and by construction that $h_0$ is eventually different from each function in $\mathcal F_0$.
Extend $h_0$ to a total function $h$ in any way such that $h$ is still eventually different each of these functions.

\medskip

It remains to show \eqref{e.eventually-diff}.
Towards a contradiction, suppose we can find $\bar x \in [\tilde p]$ and $m$ such that
for $h(n) \neq f^*(\bar x)$ for each $n \geq m$.
Since $\overline{\#}$ and $d$ are injective and by \eqref{e.front} we can find $n\in \omega$ and 
\[
\bar u \in \big(\tilde p\res (n\cap \alpha)\big)_n
\] 
such that for each $k\in n\cap \alpha$,
$\bar u(k) \subseteq \bar x(k)$
and moreover, 
$\overline{\#}\circ d^{-1}(\bar u) \geq m$.
Letting $\bar t= d^{-1}(\bar u)$,
by \eqref{e.h.catches.f} it holds that
\[
f^*(\bar x)(\overline{\#}(\bar t))=\check h(\overline{\#}(\bar t))
\]
contradicting our assumption that $f^*(\bar x)$ is different from $h$ above $m$.
\end{proof}

\section{Constructing the co-analytic Sacks-indestructible family}\label{s.main}

\begin{thm}\label{t.sigma-1-2-implies-pi-1-1}
Suppose $\mathcal E$ is a $\Sigma^1_{2}$ \emph{med} family. There is a $\Pi^1_1$ \emph{med} family $\mathcal E'$ such that for any forcing $\bbP$, if $\mathcal E$  is $\bbP$-indestructible so is $\mathcal E'$.
\end{thm}
For the proof, fix a bijection
\[
\tilde\#\colon \bigcup_{n\in\omega}({}^{n}\omega)^2\to\omega.
\]
and for each $i\in\{0,1\}$ let $c_i$ be the map from $\bbN$ to $\bigcup_{n\in\omega}({}^{n}\omega)^2$ such that
\[
m \mapsto (c_0(m),c_1(m))
\]
is the inverse function to $\tilde\#$, that is, the function such that
for all $s_0, s_1 \in \bigcup_{n\in\omega}({}^{n}\omega)^2$,
$c_i(\tilde\#(s_0,s_1)) = s_i$.
\begin{proof}[Proof of Theorem~\ref{t.sigma-1-2-implies-pi-1-1}.]
Suppose $R$ is a $\Pi^1_1$ subset of $({}^\omega\omega)^2$ such that the set $\mathcal E$ defined by
\[
\mathcal E = \{h \in {}^\omega\omega \mid (\exists z \in {}^\omega\omega)\; (h,z)\in R\}
\]
is a \emph{med} family.
By $\Pi^1_1$ uniformization, we can assume $R\cap (\mathcal E\times{}^\omega\omega)$ is the graph of a function;
of course when $h\in\mathcal E$ we can then write $R(h)$ for the unique $z$ such that $(h,z)\in R$.

Define for each pair $(h,z) \in ({}^\omega\omega)^2$  a function $g_{h,z}\colon \bbN\to\bbN$ by letting, for $n\in\bbN$
\[
g_{h,z}(n) =\begin{cases}
h(\frac{n}{2})&\text{if $n$ is even,}\\
\tilde\#(h\res n, z\res n)&\text{otherwise.}
\end{cases}
\]
Let
\begin{equation}\label{e.def-E'}
\mathcal E' = \{g_{h,R(h)} \mid h \in \mathcal E \}
\end{equation}
Clearly $\mathcal E'$ is an eventually different family.
To see $\mathcal E'$ is a \emph{med} family, suppose $f\in {}^\omega\omega$.
Consider the function $f'\in {}^\omega\omega$ defined by
\[
f'(n) = f(2n)
\]
and find $h \in \mathcal E$ such that
$\{n \in \omega \mid h(n) = f'(n) \}$ is infinite.
Thus $\{n \in \omega \mid g_{h,R(h)}(2n) = f(2n) \}$ is also infinite, so $g_{h,R(h)}$ is not eventually different from $f$, proving maximality.

\medskip

It remains to verify that $\mathcal E'$ is $\Pi^1_1$. We leave it to the reader to verify the following equivalence:
\begin{multline}
\label{e.E'} f \in \mathcal E' \Leftrightarrow \Big[ (\forall i\in\{0,1\})(\forall n,n'\in\omega)\; \big[ n < n' \Rightarrow  c_i(f(2n+1)) \subseteq c_i(f(2n'+1))\big]  \wedge\\
(\forall (h,z) \in ({}^\omega\omega)^2) \;
\big( f=g_{h,z} \Rightarrow (h,z) \in R \big)\Big]
\end{multline}
The right-hand-side of this equivalence obviously $\Pi^1_1$.

\medskip

Now let $\bbP$ be any forcing so that $\mathcal E$ is $\bbP$-indestructible and let $G$ be $(\Vee,\bbP)$-generic.
Interpreting the defining $\Sigma^1_2$ formula of $\mathcal E$ in $\Vee[G]$, we obtain a superset $\mathcal E^*$ of $\mathcal E$ which is still an eventually different family by $\Sigma^1_2$ absoluteness, and so $\mathcal E^* = \mathcal E$ because $\mathcal E$ is maximal in $\Vee[G]$.
Also by $\Sigma^1_2$ absoluteness, $R$ is a functional relation in $\Vee[G]$ and \eqref{e.def-E'} and  \eqref{e.E'} hold in $\Vee[G]$.
So the above argument also shows that $\mathcal E'$ is a \emph{med} family in $\Vee[G]$. Since $G$ was arbitrary, $\mathcal E'$ is $\bbP$-indestructible.
\end{proof}
Of course this argument generalizes to \emph{med} families with more complicated definitions provided we assume enough uniformization.

\medskip

Finally we are able prove our main result, from which also follows Theorem~\ref{t.mainmain}.

\begin{thm}\label{t.main}
There is a $\Pi^1_1$ \emph{med} family in $\eL$ which is indestructible by countable support iterations or products of Sacks forcing of any length.
\end{thm}
\begin{proof}
By Theorem~\ref{t.sigma-1-2-implies-pi-1-1} it is enough to show that there is a $\Sigma^1_2$ such family.
We work in $\eL$ and inductively construct a sequence $\langle (h_\xi, \tilde p_\xi) \mid \xi<\omega_1\rangle$
from ${}^\omega\omega\times \bbS^\omega$, where $\bbS^\omega$ denotes the full support product of $\omega$-many copies of Sacks forcing,
so that
\[
\mathcal E = \{h_\xi \mid  \xi < \omega_1\}
\]
will be our \emph{med} family.

For this purpose, first fix a sequence $\langle f_\xi \mid \xi < \omega_1\rangle$ which enumerates the set $\mathcal F$ of all codes
\[
f\colon {}^{<\omega} ({}^{<\omega} 2) \to {}^{<\omega}\omega
\]
for continuous functions ${}^\omega({}^\omega 2) \to {}^\omega\omega$ (see Section~\ref{s.prelim}).
We may in addition ask that if $\xi < \xi' < \omega_1$, it holds that $f_\xi <_{\eL} f_{\xi'}$.

Suppose we have already defined $\langle h_\xi \mid \xi < \delta\rangle$ and that
$\{h_\xi \mid \xi < \delta\}$ is an \emph{ed} family.
Let $(h_\delta,\tilde p_\delta)$ be the $\leq_{\eL}$-least pair $(h,\tilde p)\in {}^\omega\omega\times \bbS^\omega$ such that
\begin{enumerate}[label=(\Roman*), ref=(\Roman*)]
\item\label{r.h-diff} $h$ is eventually different from $h_\xi$ for each $\xi < \delta$, and
\item\label{r.the-pair} for any $\bar x \in [\tilde p]$, $h$ is not eventually different from
$f^{*}(\bar x)$.
\end{enumerate}
It is clear by Lemma~\ref{l.main.product} that such a pair $(h,\tilde p)$ exists. This finishes the definition of $\mathcal E$.

\medskip

It is clear that $\mathcal E$ is an \emph{ed} family. 
Moreover by construction, for any continuous function $f \colon {}^\omega({}^\omega 2) \to {}^\omega\omega$ there is $\tilde p \in \bbS^\omega$ and $h \in \mathcal E$ such that
 for any $\bar x \in [\tilde p]$, $h$ is not eventually different from
$f(\bar x)$.
We now show that any family $\mathcal E$ with this property is spanning\footnote{By \emph{spanning} we mean of course that
$(\forall f\in {}^\omega\omega) (\exists h \in \mathcal E)$ $f$ is not eventually different from $h$.} in any iterated Sacks extension. The proof for products is similar but simpler and is left to the reader.

Towards a contradiction let $\lambda$ be an ordinal and $\bbP$ the countably supported iteration of Sacks forcing of length $\lambda$,
and suppose we have  a $\bbP$-name $\dot f$ and $\bar p \in \bbP$  such that
\begin{equation}\label{e.main.assumption}
\bar p \forces_{\bbP}\dot f \notin \check{\mathcal E} \text{ and }\{\dot f\}\cup\check{\mathcal E}\text{ is an \emph{ed} family}.
\end{equation}
Apply Lemma~\ref{l.continuous.reading} to find $\bar q^0\leq_{\bbP}\bar p$ together with a standard enumeration $\Sigma= \langle \sigma_k \mid k\in \alpha\rangle$ of $\supp(\bar q^0)$ and a function
\[
f\colon {}^{<\alpha}({}^{<\omega} 2) \to {}^{<\omega}\omega
\]
coding a total continuous function
\[
f^* \colon {}^\alpha({}^\omega 2) \to {}^\omega\omega
\]
such that
\begin{equation}\label{e.main.contra}
\bar q^0 \forces_{\bbP} \dot f = f^* ( e^{\bar q^0,\Sigma}(\bar s_{\dot G})).
\end{equation}
Since $f \in \mathcal F$, we can find $h \in \mathcal E$
and $\bar p \in \bbS^\omega$ such that for all $\bar x \in [\tilde p]$, $h$ is not eventually different from
$f^{*}(\bar x)$.

Let $\bar q^1$ be the condition in $\bbP$ with the same support as $\bar q^0$
obtained by `pulling back' $\tilde p(k)$ via $\dot e^{\bar q^0,\Sigma}_k$ for each $k\in \alpha$, i.e., such that
for all $k\in\alpha$
\[
\forces_{\bbP}  [\bar q^1(\sigma_k)] = (\dot e^{\bar q^0,\Sigma}_k)^{-1} [[\tilde p(k)]].
\]
or equivalently, dropping superscripts from $\dot e^{\bar q^0,\Sigma}$ from now on,
\[
\forces_{\bbP} [\bar q^1] =  \dot e^{-1} [[\tilde p]].
\]
Clearly, $\bar q^1 \leq_{\bbP} \bar q^0$.

By choice of $h$ and $\tilde p$, (see Requirement~\ref{r.the-pair} above) and by absoluteness of $\Pi^1_1$ formulas
\[
\forces_{\bbP} (\forall \bar x \in [\tilde p])\; \check h\text{ and }f^{*}(\bar x)\text{ are not eventually different}
\]
and since
\[
\bar q^1 \forces_{\bbP} \dot e(\bar s_{\dot G}) \in [\tilde p]
\]
we have by Equation~\eqref{e.main.contra}
\[
\bar q^1 \forces_{\bbP} \check h\text{ and }\dot f \text{ are not eventually different}
\]
in contradiction to our assumption.
This finishes the proof that $\bbP$ forces $\check{\mathcal E}$ to be maximal.

\medskip

The theorem will be proved once we show the following claim:
\begin{clm}\label{claim.sigma12}
$\mathcal E$ is $\Sigma^1_2$.
\end{clm}
For the proof of this claim, denote by $\mathfrak S$ the effective Polish space of sequences from ${}^\omega\omega$ of length at most $\omega$,
and by $\mathfrak B$ the effective Polish space of sequences  of length at most $\omega$ from the set of perfect subtrees of ${}^{<\omega} 2$.
We may think of $\mathfrak B$ and $\mathfrak S$ as closed subsets of ${}^\omega\omega$.

Given $x\in {}^\omega2$, let $E_x \subseteq \omega^2$ be the binary relation defined by
\[
m \mathbin{E_x} n \iff x(2^m 3^n)=0.
\]
If it is the case that $E_x$ is well-founded and extensional, we denote by $M_x$ the set and by $\pi_x$ the map such that $\pi_x\colon \langle \omega, E_x\rangle \to \langle M_x, \in\rangle$ is the unique isomorphism of  $\langle \omega, E_x\rangle $ with a transitive $\in$-model.

\medskip

We also need the following well-known fact:

\begin{fct}[see {\cite[13.8]{kanamori}}]\label{fact} If $E_x$ is well-founded and extensional
and $\phi$ is a formula (in the language of set theory) with $k$ free variables,
the following relations are arithmetical in $x$:
\begin{gather*}
\{ \left( m_1,\hdots,m_k \right) \in  \omega \times\hdots \times \omega \colon \langle M_x, \in\rangle\vDash \phi(m_1,\hdots,m_k) \},\\
\{ \left(  y, m\right) \in {}^\omega\omega \times \omega \colon \pi_x(m)=y \}.
\end{gather*}
\end{fct}
We now prove the Claim.

\noindent
\textit{Proof of Claim~\ref{claim.sigma12}.}
Define a relation $P$ on ${}^\omega 2 \times {\mathfrak S}^2 \times \mathfrak B$ as follows:
Let $P(x, \vec{f}, \vec{h},\vec{p})$ if any only if
\begin{enumerate}
\item $M_x$ is well-founded and transitive and $M_x \vDash \ZF \wedge \Vee =\eL$,
\item $\vec{f}$, $\vec{h}$ and $\vec{p}$ are sequences of the same length $\alpha \leq \omega$ and for some $\{ m_{\vec{f}}, m_{\vec{h}}, m_{\vec{p}}\} \subseteq M_x$ it holds that $\pi_x(m_{\vec{f}}) = \vec{f}$, $\pi_x(m_{\vec{h}})=\vec{h}$,  and, $\pi_x(m_{\vec{p}})=\vec{p}$,
\item $\vec{f}$ enumerates an initial segment of $\mathcal F\cap M_x$,
\item\label{superitem} The following holds in $M_x$: For each $n < \alpha$,
$(\vec{h}(n),\vec{p}(n))$ is the $\leq_{\eL}$-least pair $(h,\tilde p) \in {}^\omega\omega\times\bbS^\omega$ such that
\begin{enumerate}[label=(\alph*), ref=(\ref{superitem}\alph*)]
\item\label{r.h-diff'} $h$ is eventually different from $\vec{h}(k)$ for each $k$ such that $\vec{f}(k) <_{\eL} \vec{f}(n)$ and
\item\label{r.the-pair'} for any $\bar x \in [\tilde p]$, $h$ is not eventually different from
$\vec{f}(n)^{*}(\bar x)$.
\end{enumerate}
\end{enumerate}
Note that since $M_x \vDash \ZF$ and by Mostowski Absoluteness, \ref{r.the-pair'} above holds in $M_x$ if and only if it holds in $\Vee$.

By Fact~\ref{fact}, $P$ is a $\Pi^1_1$ predicate, and $h\in \mathcal E$ if and only if for some $(x, \vec{f}, \vec{s},\vec{p})$,
$P(x, \vec{f}, \vec{s},\vec{p})$ holds and $h = \vec{h}(n)$ for some $n <\lh(\vec{h})$.
\renewcommand{\qedsymbol}{{\tiny  Claim~\ref{claim.sigma12}}~$\Box$,~{\tiny  Theorem~\ref{t.main}}~$\Box$}
\end{proof}
%\renewcommand{\qedsymbol}{}
%\end{proof}

%
%
%
%
%
%
%
%
%

\section{Property ned}\label{s.property.ned}

Clearly Lemma~\ref{l.dst} and Lemma~\ref{l.main.product} are closely related to the following property.
It is apparently a strengthening of the property of \emph{not adding an eventually different real.}
%\begin{defi}\label{d.ned}
%Let $\bbP$ be a forcing and let $\mathcal F_0$ be a countable subset  of ${}^\omega\omega$. We say $\bbP$ has property $\ned[\mathcal F_0]$ if and only for every $\bbP$-name $\dot f$ such that
%\[
%\forces_{\bbS^\alpha} \dot f\in{}^\omega\omega \text{ and $\dot f$ is eventually different from every function in $\mathcal F_0$,}
%\]
%there is  $h \in {}^\omega\omega$ which is eventually different from every function in $\mathcal{F}_0$ and $p\in\bbP$ such %that
%\begin{equation*}
%p\forces_{\bbP} \; \check h\text{ is not eventually different from }\dot f.
%\end{equation*}
%We say $\bbP$ has property $\ned$ to mean that  for every countable $\mathcal F_0 \subseteq {}^\omega\omega$, $\bbP$ %has property $\ned[\mathcal F_0]$.
%\end{defi}
\begin{defi}\label{d.ned}
We say a forcing $\bbP$ has property $\ned$ if and only for every   countable set $\mathcal F_0\subseteq{}^\omega\omega$  and every $\bbP$-name $\dot f$ such that
\[
\forces_{\bbP} \dot f\in{}^\omega\omega \text{ and $\dot f$ is eventually different from every function in $\check{\mathcal F}_0$,}
\]
there is  $h \in {}^\omega\omega$ which is eventually different from every function in $\mathcal{F}_0$ and $p\in\bbP$ such that
\begin{equation*}
p\forces_{\bbP} \; \check h\text{ is not eventually different from }\dot f.
\end{equation*}
\end{defi}

\medskip

Note again that $\{h\in{}^\omega\omega\mid h$ is eventually different from every function in $\mathcal F_0\}$ is meager in ${}^\omega\omega$.
Nevertheless, it turns out that an ${}^\omega\omega$-bounding forcing has property $\ned$ if and only if it does not add an eventually different real. 
We give a proof for completeness.
\begin{prop}\label{prop.ned}
Suppose $\bbP$ is an ${}^\omega\omega$-bounding forcing, i.e.,
\begin{equation}\label{e.no.unbounded}
\forces_{\bbP} (\forall f\in{}^\omega\omega)(\exists h\in{}^\omega\omega\cap V)(\forall n\in\omega)\;f(n) \leq h(n).
\end{equation}
 Then $\bbP$ does not add an eventually different real, i.e.,  
\begin{equation}\label{e.P.adds.no.ev.diff}
\forces_{\bbP} (\forall f\in{}^\omega\omega)(\exists h\in{}^\omega\omega\cap V)\;\text{$f$ is not eventually different from $h$.}
\end{equation}
\emph{if and only if} $\bbP$ has property $\ned$.
\end{prop}
\begin{proof}
One direction (``$\Leftarrow$'') is trivial (just set $\mathcal F_0=\emptyset$) and does not make use of \eqref{e.no.unbounded}. 

For the other direction suppose $G$ is $(V,\bbP)$-generic and working in $V[G]$, let $f \in {}^\omega\omega$ which is eventually different from every function in $\mathcal F_0$ be given.
We must find $h \in V\cap {}^\omega\omega$ which is eventually different from every function in $\mathcal F_0$ but not from $f$. 

Let $\mathcal F_0$ be enumerated as 
$\langle f_k\mid k\in\omega\rangle$ (perhaps with repetitions) and define $\langle m_k\mid k\in\omega\rangle$ by recursion as follows:
\[
m_k = \min\big\{m\in\omega\setminus \{m_l \mid l< k\} \mid(\forall n \in\omega\setminus m)\;f(n) \notin \{f_0(n),\hdots, f_k(n)\}\big\}.
\]
By assumption we can find $h^* \in {}^\omega\omega\cap V$ such that
and $h^*$ is not eventually different from $f$.
Thus letting
\[
N=\{n\in\omega\mid h^*(n)=f(n)\} 
\]
we obtain an infinite set.
Also by assumption, find $g^* \in {}^\omega\omega\cap V$ such that
\begin{equation*}\label{e.above}
(\forall n\in\omega)\;g^*(n) \geq m_k.
\end{equation*}
We may assume that $g^*$ is increasing (otherwise replace it by a faster growing function). 

\medskip

Work in $V$ until further notice.
We define $h\in{}^\omega\omega$ as follows: 
Given $n \in\omega$ fix $k\in\omega$ such that $n\in[g^*(k),g^*(k+1))$ and let
\[
h(n) = \begin{cases}
h^*(n) & \text{ if $h^*(n) \notin \{f_0(n),\hdots, f_k(n)\}$;}\\
\min\big(\omega\setminus\{f_0(n),\hdots, f_k(n)\}\big) & \text{ otherwise.}
\end{cases}
\]
Clearly $h$ is eventually different from each function in $\mathcal F_0$.
It remains to show that in $V[G]$, $f$ is not eventually different from $h$.
Given $n\in N$, fix $k\in\omega$ such that $n\in[g^*(k),g^*(k+1))$.
By choice of $g^*$, $n\geq g(k)$ and so $f(n) \notin \{f_0(n),\hdots, f_k(n)\}$.
As 
\begin{equation*}\label{e.e}
f(n) = h^*(n)
\end{equation*}
(because $n\in N$)  also $h^*(n) \notin \{f_0(n),\hdots, f_k(n)\}$, whence by construction 
\[
h(n)=h^*(n).
\] 
We have shown that $f\res N = h\res N$, whence $f$ is not eventually different from $h$.
\end{proof}

Given the previous proposition, the reader will not be surprised by the following general fact: 
\begin{prop}\label{prop.iter}
For ${}^\omega\omega$-bounding Suslin forcing, property $\ned$ is preserved under countable support iterations. 
\end{prop}
Recall here that a forcing $\bbP$ is said to be Suslin if and only if it has a presentation $\langle \bbP', \leq_{\bbP'}\rangle$ such that $\bbP' \subseteq {}^\omega\omega$, $\leq_{\bbP'} \subseteq ({}^\omega\omega)^2$ and $\bbP'$, $\leq_{\bbP'}$, and the incompatibility relation are all $\mathbf{\Sigma^1_1}$. For the definition of \emph{proper} forcing see, e.g., \cite{handbook_proper}.

\begin{proof}[Proof of Proposition~\ref{prop.iter}.]
We have seen in Proposition~\ref{prop.ned} that for ${}^\omega\omega$-bounding forcing, property $\ned$ is equivalent to not adding an eventually different real.
The proof of \cite[2.4.8,~p.~59]{barto} shows that any forcing has the latter property if and only if it does not force that ${}^\omega\omega\cap V$ is meager.
 
Furthermore, these equivalent properties are known to be preserved under countable support iterations for ${}^\omega\omega$-bounding Suslin forcings:
For this class of forcings, by \cite[2.2]{preserving} not forcing ${}^\omega\omega\cap V$ to be meager is equivalent to preserving non-meagerness, which for all ${}^\omega\omega$-bounding proper forcings is preserved under countable support iterations by \cite[6.3.20~and~6.3.21]{barto}.
\end{proof}

While this implies that countably supported iterations of Sacks forcing have property ned (since single Sacks forcing has it by Lemma~\ref{l.sacks-ev-diff} and Proposition~\ref{prop.ned}), more is true:
\begin{prop}\label{prop.general-case}
Any countable support iteration or product  (of any length)  of Sacks forcing has property $\ned$.
\end{prop}
We want to point out that using Lemma~\ref{l.continuous.reading} (continuous reading of names), one can easily turn the proof of Lemma~\ref{l.main.product} into an elementary and direct proof of Proposition~\ref{prop.general-case} in full generality (for iterations \emph{and} products).
\begin{proof}[Proof of Proposition~\ref{prop.general-case}]
Suppose $\bbP$ is a countably supported iteration or product of Sacks forcing. Let a countable set $\mathcal F_0\subseteq {}^\omega\omega$ and a $\bbP$-name $\dot f$ as in Proposition~\ref{prop.general-case} be given.
Just as in the proof of Theorem~\ref{t.main} below, use Lemma~\ref{l.continuous.reading} to find $\bar q \in \bbP$ together with a standard enumeration $\Sigma$ of length $\alpha\leq\omega$ of $\supp(\bar q)$ and a function
$
f\colon {}^{<\alpha}({}^{<\omega} 2) \to {}^{<\omega}\omega
$
coding a total continuous function
$
f^* \colon {}^\alpha({}^\omega 2) \to {}^\omega\omega
$
such that
$
\bar q \forces_{\bbP} \dot f = f^* (e^{\bar q^0,\Sigma}(\bar s_{\dot G})).
$

As in the proof of Lemma~\ref{l.continuous.reading}, we can find $h \in {}^\omega\omega$ eventually different from every function in $\mathcal F_0$ and $\tilde p \in \bbS^\alpha$ such that 
for any $x \in [\tilde p]$, $f^*(x)$ and $h$ are not eventually different.
Argue exactly as below in the proof of Theorem~\ref{t.main} that using $e^{\bar q^0,\Sigma}$ we can obtain a condition $\bar p \in \bbP$ from $\tilde p$ such that by $\mathbf{\Pi}^1_1$ absoluteness, $\bar p \forces_{\bbP} \check h$ is not eventually different from $\dot f$.
\end{proof}

\section{Open Questions}\label{s.open}

We close with some open questions. 
Here, for a pointclass $\Gamma\subseteq{}^\omega\omega$, $\mathfrak{a}^{\Gamma}_g$ means the minimum size of a maximal cofinitary group in $\Gamma$, $\mathfrak{a}^{\Gamma}_p$ means the minimum size in the Polish space $S_\infty$ (i.e., the space of permutations of $\omega$) of a maximal  eventually different family  in $\Gamma$, and $\mathfrak{a}^{\Gamma}$ means the minimum size of a \emph{mad} family in $\Gamma$. 
When $\Gamma$ is omitted of course it is taken to be $\Power(\omega)$.
See \cite{blass} or \cite{barto} for definitions of the cardinal invariants $\mathfrak{s}$ and $\operatorname{\mathbf{non}}(\mathcal M)$ mentioned below.

It is known that consistently, $\mathfrak{a} < \mathfrak{a}_g, \mathfrak{a}_p$.
Indeed, since $\mathbf{non}(\mathcal M) \leq \mathfrak{a}_g, \mathfrak{a}_p$ and $\mathfrak{s}\leq\mathbf{non}(\mathcal M)$, this holds in any model of $\mathfrak{a} < \mathfrak{s}$ (see for example \cite{brendle_fischer_2011}).
\begin{enumerate}
\item What strict relations are consistent between $\mathfrak{a}_g$, $\mathfrak{a}_p$ and $\mathfrak{a}$?
\item What about $\mathfrak{a}^{\mathbf{\Pi}^1_1}_g$, $\mathfrak{a}^{\mathbf{\Pi}^1_1}_p$, and $\mathfrak{a}^{\mathbf{\Pi}^1_1}$?
\end{enumerate}
Finally, we ask if property $\ned$ is distinct from ``not adding an eventually different real'' in general, i.e., for forcings which are not ${}^\omega\omega$-bounding.

\bibliographystyle{amsplain}
\bibliography{eventually_different}

\providecommand{\bysame}{\leavevmode\hbox to3em{\hrulefill}\thinspace}
\providecommand{\MR}{\relax\ifhmode\unskip\space\fi MR }
% \MRhref is called by the amsart/book/proc definition of \MR.
\providecommand{\MRhref}[2]{%
  \href{http://www.ams.org/mathscinet-getitem?mr=#1}{#2}
}
\providecommand{\href}[2]{#2}
\begin{thebibliography}{10}

\bibitem{handbook_proper}
Uri Abraham, \emph{Proper forcing}, Handbook of set theory. {V}ols. 1, 2, 3,
  Springer, Dordrecht, 2010, pp.~333--394. \MR{2768684}

\bibitem{barto}
Tomek Bartoszy\'nski and Haim Judah, \emph{Set theory - on the structure of the
  real line}, A K Peters, Ltd., Wellesley, MA, 1995. \MR{1350295}

\bibitem{blass}
Andreas Blass, \emph{Combinatorial cardinal characteristics of the continuum},
  Handbook of set theory. {V}ols. 1, 2, 3, Springer, Dordrecht, 2010,
  pp.~395--489. \MR{2768685}

\bibitem{JBVFYK}
Joerg Brendle, Vera Fischer, and Yurii Khomskii, \emph{Definable maximal
  independnet families}, preprint, 2018.

\bibitem{JBYK}
Joerg Brendle and Yurii Khomskii, \emph{Mad families constructed from perfect
  almost disjoint families}, Journal of Symbolic Logic \textbf{78} (2013),
  no.~4, 1164--1180.

\bibitem{brendle_fischer_2011}
J{\"o}rg Brendle and Vera Fischer, \emph{Mad families, splitting families and
  large continuum}, The Journal of Symbolic Logic \textbf{76} (2011), no.~1,
  198--208.

\bibitem{VFSFYK}
Vera Fischer, Sy~Friedman, and Yurii Khomskii, \emph{Co-analytic mad families
  and definable wellorders}, Archive for Mathematical Logic \textbf{52} (2013),
  no.~7-8, 809--822.

\bibitem{VFSFDSAT}
Vera Fischer, Sy~Friedman, David Schrittesser, and Asger T{\"o}rnquist,
  \emph{Definable maximal cofinitary group of intermediate sizes}, preprint,
  2018.

\bibitem{VFSFATOM}
Vera Fischer, Sy~Friedman, and Asger T{\"o}rnquist, \emph{Projective maximal
  families of orthogonal measures with large continuum}, Journal of Logic and
  Analysis \textbf{4} (2012), 1--15.

\bibitem{VFSFLZ}
Vera Fischer, Sy~Friedman, and Lyubomyr Zdomskyy, \emph{Projective wellorders
  and mad families with large continuum}, Annals of Pure and Applied Logic
  \textbf{162} (2011), 853--862.

\bibitem{VFDM}
Vera Fischer and Diana Montoya, \emph{Ideals of independence}, preprint, 2017.

\bibitem{VFDSAT}
Vera Fischer, David Schrittesser, and Asger T{\"o}rnquist, \emph{A co-analytic
  {C}ohen indestructible maximal cofinitary group}, Journal of Symbolic Logic
  \textbf{82} (2017), no.~2, 629--247.

\bibitem{VFATOM}
Vera Fischer and Asger T{\"o}rnquist, \emph{A co-analytic maximal set of
  orthogonal measures}, Journal of Symbolic Logic \textbf{75} (2010), no.~4,
  1403--1414.

\bibitem{SFLZ}
Sy~Friedman and Lyubomyr Zdomskyy, \emph{Projective mad families}, Annals of
  Pure and Applied Logic \textbf{161} (2010), no.~12, 1581--1587.

\bibitem{borel-mcg}
Haim Horowitz and Saharon Shelah, \emph{A {B}orel maximal cofinitary group},
  \href{https://arxiv.org/abs/1610.01344}{\nolinkurl{arXiv:1610.01344
  [math.LO]}}, Oct 2016.

\bibitem{medf-borel}
\bysame, \emph{A {B}orel maximal eventually different family},
  \href{https://arxiv.org/abs/1605.07123}{\nolinkurl{arXiv:1605.07123
  [math.LO]}}, May 2016.

\bibitem{jech-set}
Thomas Jech, \emph{Set theory}, Springer Monographs in Mathematics,
  Springer-Verlag, Berlin, 2003, The third millennium edition, revised and
  expanded. \MR{1940513 (2004g:03071)}

\bibitem{kanamori}
Akihiro Kanamori, \emph{The higher infinite - {Large} cardinals in set theory
  from their beginnings}, second ed., Springer Monographs in Mathematics,
  Springer-Verlag, Berlin, 2003. \MR{1994835}

\bibitem{kechris1995}
Alexander~S. Kechris, \emph{Classical descriptive set theory}, Graduate Texts
  in Mathematics, vol. 156, Springer-Verlag, New York, 1995. \MR{1321597}

\bibitem{preserving}
Jakob Kellner and Saharon Shelah, \emph{Preserving preservation}, The Journal
  of Symbolic Logic \textbf{70} (2005), no.~3, 914--945.

\bibitem{mathias.1967}
A.~R.~D. Mathias, \emph{Happy families}, Ann. Math. Logic \textbf{12} (1977),
  no.~1, 59--111. \MR{0491197}

\bibitem{miller-infinite}
Arnold~W. Miller, \emph{Infinite combinatorics and definability}, Ann. Pure
  Appl. Logic \textbf{41} (1989), no.~2, 179--203. \MR{983001 (90b:03070)}

\bibitem{moschovakis}
Yiannis~N. Moschovakis, \emph{Descriptive set theory}, second ed., Mathematical
  Surveys and Monographs, vol. 155, American Mathematical Society, Providence,
  RI, 2009. \MR{2526093}

\bibitem{schrittesser-mof-large-continuum}
David Schrittesser, \emph{Definable discrete sets and large continuum},
  {\href{http://arxiv.org/abs/1610.03331}{\nolinkurl{arXiv:1610.03331}}}.

\bibitem{on-shelah}
\bysame, \emph{On {H}orowitz and {S}helah's maximal eventually different
  family},
  {\href{http://arxiv.org/abs/arXiv:1703.01806}{\nolinkurl{arXiv:1703.01806
  [math.LO]}}}, March 2017.

\bibitem{medf-bounded}
\bysame, \emph{Compactness of maximal eventually different families}, Bulletin
  of the London Mathematical Society \textbf{50} (2018), no.~2, 340--348.

\bibitem{mof}
David Schrittesser and Asger T{\"o}rnquist, \emph{Definable maximal discrete
  sets in forcing extensions},
  {\href{http://arxiv.org/abs/1510.08781}{\nolinkurl{arXiv:1510.08781}}}. To
  appear in {\it Mathematical Research Letters}.

\bibitem{She92}
Saharon Shelah, \emph{$\hbox{Con}(\mathfrak{i}<\mathfrak{u})$}, Archive for
  Mathmatical Logic \textbf{31} (1992), no.~6, 433--443.

\end{thebibliography}

\end{document}